\documentclass[a4paper,10pt]{article}
\usepackage[latin1]{inputenc}
\usepackage[T1]{fontenc}
\usepackage{amsmath}
\usepackage{amsfonts}
\usepackage{amstext}
\usepackage{amsthm}
\usepackage[all]{xy}
\usepackage{geometry}
\usepackage{srcltx}
\usepackage{graphics}
\usepackage{epsfig}
\usepackage{subfigure}
\usepackage{slashed}
\usepackage{color}
\usepackage{amssymb}
\usepackage{srcltx}
\newtheorem{theorem}{Theorem}[section]

\newtheorem{defi}[theorem]{Definition}

\newtheorem{prop}[theorem]{Proposition}

\DeclareMathOperator{\im}{im}

\DeclareMathOperator{\sfl}{sf}

\DeclareMathOperator{\sgn}{sgn}

\DeclareMathOperator{\Hom}{Hom}
\DeclareMathOperator{\GL}{GL}

\DeclareMathOperator{\Mat}{Mat}

\title{Local and Global Bifurcation for Periodic Solutions of Hamiltonian Systems via Comparison Theory for the Spectral Flow}
 
\author{Joanna Janczewska, Maciej Starostka and Nils Waterstraat}

\begin{document}
\date{}
\maketitle


\begin{abstract}
We obtain local and global bifurcation for periodic solutions of Hamiltonian systems by using a new way to apply a comparison principle of the spectral flow that was originally introduced by Pejsachowicz in a joint work with the third author. A particular novelty is the study of global bifurcation, which to the best of our knowledge has not been done via the spectral flow.  
\end{abstract}

\section{Introduction}
We consider Hamiltonian systems

\begin{equation}\label{equation}
\left\{
\begin{aligned}
J u'(t)+\nabla_u&\mathcal{H}(\lambda,t,u(t))=0,\quad t\in [0,2\pi]\\
u(0)&=u(2\pi),
\end{aligned}
\right.
\end{equation}  
where $J$ denotes the standard symplectic matrix

\begin{align}\label{J}
J=\begin{pmatrix}
0&-I_n\\
I_n&0
\end{pmatrix}
\end{align}
and $\mathcal{H}:\mathbb{R}\times\mathbb{R}\times\mathbb{R}^{2n}\rightarrow\mathbb{R}$ is $C^2$. We assume that $\mathcal{H}(\lambda,t,u)$ is $2\pi$-periodic with respect to $t$, that $\nabla_u\mathcal{H}(\lambda,t,0)=0$ for all $(\lambda,t)\in\mathbb{R}\times\mathbb{R}$ and consider bifurcation of periodic solutions from the trivial one $u\equiv 0$. Such problems have been studied by various methods for several decades and are well presented, e.g., in the monographs \cite{Bartsch}, \cite{Mawhin} and \cite{Rabinowitz}. In 1999 Fitzpatrick, Pejsachowicz and Recht showed in \cite{Specflow} that the spectral flow, which is an invariant from global analysis, yields a novel and promising way to find bifurcation points of critical points of families of functionals. In their subsequent work \cite{SFLPejsachowiczII} they applied their findings to the Hamiltonian systems \eqref{equation}.\\
The spectral flow was invented by Atiyah, Patodi and Singer in 1976 for studying spectra of elliptic differential operators on manifolds in \cite{AtiyahPatodi} and quickly became a valuable invariant in various settings in analysis, geometry and mathematical physics \cite{book}, \cite{Fredholm}. Since the work \cite{Specflow} it has been applied to study bifurcation problems for different kinds of differential equations by various authors. Pejsachowicz and the third author invented in \cite{BifJac} a comparison principle that can be used to estimate the spectral flow, and which allows to find bifurcation points for \eqref{equation} just from the eigenvalues of the family of Hessians of the Hamiltonian $\mathcal{H}$ at $0\in\mathbb{R}^{2n}$.\\
The aim of this paper is threefold. Firstly, we give a rather short and novel proof of a bifurcation criterion for \eqref{equation} that is based on the comparison principle and was essentially already obtained by different methods in \cite{BifJac}. Secondly, we refine this new proof to find bifurcation of \eqref{equation} under far less restrictive assumptions on the eigenvalues of the family of Hessians of the Hamiltonian $\mathcal{H}$ at $0\in\mathbb{R}^{2n}$. Let us emphasize that we here need to make an additional assumption on $\mathcal{H}$, so that this second part is not a direct generalisation of the first one. Finally, we use the method from the previous part to consider global bifurcation of \eqref{equation}, which is of particular interest as an example of B\"ohme \cite{Boehme} indicates that there is little hope to find global bifurcation by variational methods. Actually, the possibility of a global bifurcation was not even mentioned in the original work \cite{Specflow} by Fitzpatrick, Pejsachowicz and Recht, and we are not aware of any reference in the literature except \cite{Benevieri}, where the authors explicitly state that the spectral flow seems to be unable to provide global bifurcation results. Our final aim is to show that this statement is despite B\"ohme's example too pessimistic. Here we recap Fitzpatrick and Pejsachowicz' construction of the parity, which can show bifurcation for non-linear operator equations in Banach spaces. In Hilbert spaces the parity is linked to the spectral flow and thus opens up the way to find global bifurcation points for critical points of families of functionals. We are not aware of any reference in the literature for the obtained Theorem \ref{thm-globalbif}, but we are sure that its statement is folklore. Finally, we combine this theorem with the comparison methods from the previous parts and obtain global bifurcation for planar systems \eqref{equation} just from the eigenvalues of the Hessians of the Hamiltonian $\mathcal{H}$ at $0\in\mathbb{R}^{2n}$.\\                         
The paper is structured as follows. In the next section we recall the variational formulation from \cite{Rabinowitz} to find solutions of \eqref{equation}, give the central definitions and state our main theorem. The following Section 3 deals with the proof of this theorem and is divided into four parts. Firstly, we recall the construction of the parity and its application to bifurcation for nonlinear operator equations in Banach spaces. Secondly, we recap the spectral flow and how it is used to find bifurcation points of critical points of families of functionals in Hilbert spaces. Section 3.3 joins the parity and the spectral flow in Proposition \ref{prop-parity} and yields Theorem \ref{thm-globalbif} on global bifurcation of branches of critical points. As already mentioned above, we are not aware of any reference in the literature for this theorem, but consider it as folklore. The final Section 3.4 deals with the proof of our main theorem, which consists of three parts.

\section{The Main Theorem}
Let $\mathcal{H}:\mathbb{R}\times\mathbb{R}\times\mathbb{R}^{2n}\rightarrow\mathbb{R}$ be $C^2$ and assume that $\mathcal{H}(\lambda,t,u)$ is $2\pi$-periodic with respect to $t$ and $\nabla_u\mathcal{H}(\lambda,t,0)=0$ for all $(\lambda,t)\in\mathbb{R}\times\mathbb{R}$. Let us consider as in \eqref{equation} the Hamiltonian systems

\begin{equation*}
\left\{
\begin{aligned}
J u'(t)+\nabla_u&\mathcal{H}(\lambda,t,u(t))=0,\quad t\in [0,2\pi]\\
u(0)&=u(2\pi),
\end{aligned}
\right.
\end{equation*}  
where $J$ denotes the standard symplectic matrix \eqref{J}, and note that under the given assumptions $u\equiv 0$ is a solution of \eqref{equation} for all $\lambda\in\mathbb{R}$.

\begin{defi}
Let $\lambda_-<\lambda_+$ be two real numbers. 

\begin{itemize}
 \item[(i)] There is a (local) bifurcation point $\lambda_0\in[\lambda_-,\lambda_+]$ for \eqref{equation}, if in every neighbourhood of $(\lambda_0,0)$ in $\mathbb{R}\times C^1(S^1,\mathbb{R}^{2n})$ there is an element $(\lambda,u)$ such that $u\neq 0$ is a solution of \eqref{equation}.
 \item[(ii)] Let $\mathcal{S}$ denote the closure of $$\{(\lambda,u)\in\mathbb{R}\times C^1(S^1,\mathbb{R}^{2n}):\, u\, \text{is a solution of}\, \eqref{equation}\}\setminus\mathbb{R}\times\{0\}$$ in $\mathbb{R}\times H^\frac{1}{2}(S^1,\mathbb{R}^{2n})$, and let $\mathcal{T}$ be the connected component of $\mathcal{S}\cup[\lambda_-,\lambda_+]\times\{0\}$ containing $[\lambda_-,\lambda_+]\times\{0\}$. Then there is a global bifurcation from the interval $[\lambda_-,\lambda_+]$ if either $\mathcal{T}$ is unbounded or $\mathcal{T}$ contains a point $(\lambda^\ast,0)$ such that $\lambda^\ast\notin[\lambda_-,\lambda_+]$.
\end{itemize}
\end{defi}
\noindent
Before stating the main theorem of this work, we now briefly recap the analytical setting for finding solutions of \eqref{equation} by variational methods, where we follow the presentation of \cite{SFLPejsachowiczII} and \cite{Rabinowitz}.\\
The Hilbert space $H^\frac{1}{2}(S^1,\mathbb{R}^{2n})$ consists of all functions $u:[0,2\pi]\rightarrow\mathbb{R}^{2n}$ such that
\begin{align*}
u(t)=c_0+\sum^\infty_{k=1}{(a_k\sin\,kt+b_k\cos\,kt)},
\end{align*}
where $c_0,a_k,b_k\in\mathbb{R}^{2n}$, $k\in\mathbb{N}$, and 
$\sum^\infty_{k=1}{k(|a_k|^2+|b_k|^2)}<\infty.$\\
The scalar product on $H^\frac{1}{2}(S^1,\mathbb{R}^{2n})$ is defined  by
\begin{align*}
\langle u,v\rangle_{H^\frac{1}{2}}=\langle c_0,\tilde{c}_0\rangle+\sum^\infty_{k=1}{k(\langle a_k,\tilde{a}_k\rangle+\langle b_k,\tilde{b}_k\rangle)},
\end{align*}
where $\tilde{c}_0$ and $\tilde{a}_k,\tilde{b}_k$ denote the Fourier coefficients of $v\in H^\frac{1}{2}(S^1,\mathbb{R}^{2n})$. \\
Let
\begin{align}\label{Gamma}
Q:H^\frac{1}{2}(S^1,\mathbb{R}^{2n})\times H^\frac{1}{2}(S^1,\mathbb{R}^{2n})\rightarrow\mathbb{R}
\end{align}
be the unique continuous extension of the bounded bilinear form
\[\tilde{Q}(u,v)=\int^{2\pi}_0{\langle J u'(t),v(t)\rangle\,dt}\]
defined on the dense subspace $H^1(S^1,\mathbb{R}^{2n})$ of $H^\frac{1}{2}(S^1,\mathbb{R}^{2n})$, which consists of all absolutely continuous functions $u:S^1\rightarrow\mathbb{R}^{2n}$ having a square integrable derivative. We consider the map
\[\psi:\mathbb{R}\times H^\frac{1}{2}(S^1,\mathbb{R}^{2n})\rightarrow\mathbb{R},\quad \psi_\lambda(u)=\frac{1}{2}\,Q(u,u)+\int^{2\pi}_0{\mathcal{H}(\lambda,t,u(t))\,dt}.\] 
It is a standard result (see \cite{Rabinowitz}, \cite{book}) that $\psi$ is $C^2$ if there are constants $a,b\geq 0$ and $r>1$ such that
	 \begin{align*}
\begin{split}
|\nabla_u\mathcal{H}(\lambda,t,u)|&\leq a+b|u|^r,\\
|D_u\nabla_u\mathcal{H}(\lambda,t,u)|&\leq a+b|u|^r,\quad (\lambda,t,u)\in  \mathbb{R}\times\mathbb{R}\times\mathbb{R}^{2n}.
\end{split}
\end{align*}
Moreover, 
\[\langle \nabla_u\psi_\lambda  u,v\rangle_{H^\frac{1}{2}}=Q(u,v)+\int^{2\pi}_0{\langle\nabla_u\mathcal{H}(\lambda,t,u(t)),v(t)\rangle\,dt},\quad v\in H^\frac{1}{2}(S^1,\mathbb{R}^{2n}),\]
and consequently, the critical points of $\psi_\lambda$ are the weak solutions of the Hamiltonian system \eqref{equation}. In particular,$ \ u\equiv 0\in H^\frac{1}{2}(S^1,\mathbb{R}^{2n})$ is a critical point of  $\psi_\lambda,$   whose Hessian is given by
\[\langle L_\lambda u,v\rangle_{H^\frac{1}{2}}=Q(u,v)+\int^{2\pi}_0{\langle A_\lambda(t) u(t),v(t)\rangle\,dt},\]
where $A_\lambda(t)$ is the Hessian matrix of $\mathcal{H}_\lambda(t,\cdot):\mathbb{R}^{2n}\rightarrow\mathbb{R}$ at $0$. Note that the kernel of $L_\lambda$ is made by the solutions of the linear equation

\begin{equation}\label{equationlin}
\left\{
\begin{aligned}
J u'(t)&+A_\lambda(t)u(t)=0,\quad t\in [0,2\pi]\\
u(0)&=u(2\pi).
\end{aligned}
\right.
\end{equation}  
We now let $\mu_1(\lambda,t)\leq\cdots\leq\mu_{2n}(\lambda,t)$ be the eigenvalues of the matrices $A_\lambda(t)$ for $\lambda\in\mathbb{R}$ and $t\in[0,2\pi]$. For two fixed real numbers $\lambda_-<\lambda_+$ we set

\begin{align}\label{alpha}
\alpha_\pm:=\inf_{t\in[0,2\pi]}{\mu_1(\lambda_\pm,t)}
\end{align}
as well as

\begin{align}\label{beta}
\beta_\pm:=\sup_{t\in[0,2\pi]}{\mu_{2n}(\lambda_\pm,t)}.
\end{align}
Finally, we define for any real numbers $\mu,\nu\in\mathbb{R}$

\[\Delta(\mu,\nu)=\begin{cases}
|\{i\in\mathbb{Z}:\, \mu< i\leq\nu\}|,\quad\text{if}\, \mu\leq\nu\\
-|\{i\in\mathbb{Z}:\, \nu\leq i<\mu\}|,\quad\text{if}\, \nu\leq\mu,
\end{cases}\]
where $|\{\cdot\}|$ stands for the cardinality of a set. Now our main theorem is as follows. 
\begin{theorem}\label{thm-main}
Let $\mathcal{H}:\mathbb{R}\times\mathbb{R}\times\mathbb{R}^{2n}\rightarrow\mathbb{R}$ be a family of Hamiltonians as above and $\alpha_\pm,\beta_\pm$ as in \eqref{alpha} and \eqref{beta}. Assume that \eqref{equationlin} only has the trivial solution for $\lambda=\lambda_\pm$.

\begin{enumerate}
 \item[(i)] If $\beta_-<\alpha_+$ and $\Delta(\beta_-,\alpha_+)> 0$, or $\beta_+<\alpha_-$ and $ \Delta(\alpha_-,\beta_+)< 0,$ then there is a (local) bifurcation point of \eqref{equation} in $(\lambda_-,\lambda_+)$.
 \item[(ii)] If $\alpha_-<0<\beta_-$, $A_{\lambda_-}$ is autonomous, and either $\beta_-<\alpha_+$ or $\beta_+<\alpha_-$, then there is a (local) bifurcation point of \eqref{equation} in $(\lambda_-,\lambda_+)$.
 \item[(iii)] If $n=1$ in $(ii)$, i.e. \eqref{equation} is planar, and 
 
 \begin{align}\label{Delta}
 \Delta(\alpha_-,\beta_+)-\Delta(\beta_-,\alpha_+)=1,
 \end{align}
then there is a global bifurcation of \eqref{equation} from the interval $[\lambda_-,\lambda_+]$.  
\end{enumerate}
\end{theorem}
\noindent
Let us emphasise that a proof of $(i)$ is implicitly contained in \cite[\S 8]{BifJac}, albeit here we use a different argument. Nevertheless, the essential novelties of our work are $(ii)$ and $(iii)$. Note that $(ii)$ allows the gap between $\beta_\mp$ and $\alpha_\pm$ to be arbitrarily small under additional assumptions on $A_{\lambda_-}$, whereas the gap in $(i)$ might need to be rather large to find bifurcation. To the best of our knowledge, $(iii)$ is the first result that shows global bifurcation by a spectral flow argument.

\section{Proof of Theorem \ref{thm-main}}
As announced in the introduction, the proof of $(iii)$ requires various preliminaries which however yield an abstract theorem on global bifurcation that is of independent interest. Thus we firstly recap the parity and the spectral flow as well as their relevance in bifurcation theory in the next two subsection. Afterwards we join these two notions and introduce an abstract global bifurcation theorem. In the fourth subsection we finally consider the Hamiltonian systems \eqref{equation} and show Theorem \ref{thm-main}.

\subsection{The Parity and Global Bifurcation}\label{sec-parity}
Let us first recap the definition of the parity, where we mostly follow \cite{Memoirs}. Let $E$ be a real Banach space of infinite dimension and let $\Phi_0(E)$ denote the open subset of $\mathcal{L}(E)$ of all Fredholm operators of index $0$ with the norm topology. Let $L:[a,b]\rightarrow\Phi_0(E)$ be a path such that $L_a$ and $L_b$ are invertible. Fitzpatrick and Pejsachowicz showed in \cite{Mike} that there is a \textit{parametrix} for $L$, i.e., a path $M:[a,b]\rightarrow\GL(E)$ such that $M_\lambda L_\lambda=I_E+K_\lambda$ for some path $K:[a,b]\rightarrow\mathcal{K}(E)$ of compact operators. As $I_E+K_\lambda$ are invertible for $\lambda=a,b$, the Leray-Schauder degree of these operators is defined and given by

\begin{align}\label{LSdeg}
\deg_{LS}(I_E+K_\lambda)=(-1)^{k(\lambda)},
\end{align}   
where $k(\lambda)$ denotes the algebraic multiplicity of the number of eigenvalues less than $-1$ of $K_\lambda$. We now define $\sigma(L,[a,b])$ as the unique element in $\mathbb{Z}_2=\{0,1\}$ such that

\[\deg_{LS}(I_E+K_a)\deg_{LS}(I_E+K_b)=(-1)^{\sigma(L,[a,b])}.\]
It was proved in \cite{Mike} that this element of $\mathbb{Z}_2$ does not depend on the choice of the parametrix $M$ and Fitzpatrick and Pejsachowicz introduced $\sigma$ as the parity of the path $L$ in the latter reference. Moreover, they showed that the parity has the following properties:

\begin{itemize}
 \item[(P1)] If $L_\lambda\in\GL(E)$ for all $\lambda\in[a,b]$, then $\sigma(L,[a,b])=0$.
 \item[(P2)] If $c\in[a,b]$ such that $L_c\in\GL(E)$, then $\sigma(L,[a,b])=\sigma(L,[a,c])+\sigma(L,[c,b])$.
 \item[(P3)] If $h:[a,b]\times[0,1]\rightarrow\Phi_0(E)$ is a homotopy such that $h(a,s), h(b,s)\in\GL(E)$ for all $s\in[0,1]$, then 
 
 \[\sigma(h(\cdot,0),[a,b])=\sigma(h(\cdot,1),[a,b]).\]
 \item[(P4)] If $E=E_1\oplus E_2$ for two closed subspaces of $E$ such that $L_\lambda(E_i)\subset E_i$ for all $\lambda\in[a,b]$, $i=1,2$, then
 
 \[\sigma(L,[a,b])=\sigma(L\mid_{E_1},[a,b])+\sigma(L\mid_{E_2},[a,b]).\]
 \end{itemize}
Note that by (P2) and (P3) the parity induces a homomorphism $\sigma:\pi_1(\Phi_0(E))\rightarrow\mathbb{Z}_2$, which is bijective under an additional assumption:
 
\begin{itemize}
 \item[(P5)] If $\GL(E)$ is contractible with respect to the norm topology, then 
 
 \[\sigma:\pi_1(\Phi_0(E))\rightarrow\mathbb{Z}_2\]
 is an isomorphism. 
\end{itemize}
Consequently, by Kuiper's Theorem, $\sigma$ in particular induces an isomorphism if $E$ is a Hilbert space.\\
The main motivation for introducing the parity is bifurcation theory of nonlinear operator equations. Below we will make use of the following theorem that basically can already be found in \cite{Mike}, but which in this generality firstly appeared in \cite{Rabier}.

\begin{theorem}\label{Rabier}
Let $X,Y$ be real Banach spaces and $G:\mathbb{R}\times X\rightarrow Y$ a $C^1$-map with $G(\lambda,0)=0$ for all $\lambda\in\mathbb{R}$. Suppose that the derivatives $D_0G_\lambda$ of $G_\lambda:X\rightarrow Y$ at $0\in X$ are Fredholm of index $0$, and that there are $\lambda_-<\lambda_+$ such that $D_0G_{\lambda_\pm}\in\GL(X,Y)$ as well as $\sigma(D_0G_\cdot,[\lambda_-,\lambda_+])=1$.\\
Let $\mathcal{S}$ denote the closure of $G^{-1}(0)\setminus\mathbb{R}\times\{0\}$ in $\mathbb{R}\times X$ and $\mathcal{T}$ the connected component of $\mathcal{S}\cup[\lambda_-,\lambda_+]\times\{0\}$ containing $[\lambda_-,\lambda_+]\times\{0\}$. Then either $\mathcal{T}$ is non-compact or $\mathcal{T}$ contains a point $(\lambda^\ast,0)$ with $\lambda^\ast\notin[\lambda_-,\lambda_+]$.   
\end{theorem}
\noindent
Note that if $G$ is proper on any bounded subset of $\mathbb{R}\times X$, then the non-compactness of $C$ implies that $C$ is unbounded in $\mathbb{R}\times X$. 

\subsection{The Spectral Flow and Bifurcation of Critical Points}
Let us now consider a real separable Hilbert space $H$ and denote by $\Phi_S(H)$ the subset of $\mathcal{L}(H)$ of all selfadjoint Fredholm operators with the norm topology. We follow the definition of the spectral flow from \cite{Phillips} (see also \cite[\S 4.1]{book}) and note at first that if $0$ is in the spectrum of a selfadjoint Fredholm operator, then it is an isolated eigenvalue of finite multiplicity. This allows to find a partition $a=\lambda_0<\lambda_1<\cdots<\lambda_N=b$ and numbers $a_i>0$, $i=1,\ldots,N$ such that the spectral projection $\chi_{[-a_i,a_i]}(L_\lambda)$ has constant finite rank for $\lambda\in[\lambda_{i-1},\lambda_i]$ and $i=1,\ldots N$. Then the spectral flow of $L:[a,b]\rightarrow\Phi_S(H)$ is defined by

\begin{align}\label{def-sfl}
\sfl(L,[a,b])=\sum^N_{i=1}(\dim\im(\chi_{[0,a_i]}(L_{\lambda_i})))-\dim\im(\chi_{[0,a_i]}(L_{\lambda_{i-1}})))\in\mathbb{Z},
\end{align}
and it was shown by Phillips in \cite{Phillips} that it neither depends on the choice of the numbers $\lambda_i$ nor the $a_i$. The spectral flow has the following properties:

\begin{enumerate}
 \item[(S1)] If $L_\lambda\in\GL(H)\cap\Phi_S(H)$, then $\sfl(L,[a,b])=0$.
 \item[(S2)] If $c\in[a,b]$, then $\sfl(L,[a,b])=\sfl(L,[a,c])+\sfl(L,[c,b])$.
 \item[(S3)] If $h:[a,b]\times[0,1]\rightarrow\Phi_S(H)$ is such that $h(a,s)$, $h(b,s)$ are constant for all $s\in[0,1]$, then
 
 \[\sfl(h(\cdot,0))=\sfl(h(\cdot,1)).\]
 If $h(a,s)$, $h(b,s)$ are non-constant, then the homotopy invariance also holds as long as $h(a,s)$, $h(b,s)$ are invertible for all $s\in[0,1]$ (cf. \cite{CompSfl}).
 \item[(S4)] If $H=H_1\oplus H_2$ for two closed subspaces of $H$ that reduce $L_\lambda$, $\lambda\in[a,b]$, then
 
 \[\sfl(L,[a,b])=\sfl(L\mid_{H_1},[a,b])+\sfl(L\mid_{H_2},[a,b]).\]
 \end{enumerate}
Atiyah and Singer showed in \cite{AtiyahSinger} that $\Phi_S(H)$ consists of the three path components

\begin{align*}
\Phi^+_S(H)&=\{T\in\Phi_S(H):\,\sigma_{ess}(T)\subset(0,\infty)\},\\
\Phi^-_S(H)&=\{T\in\Phi_S(H):\,\sigma_{ess}(T)\subset(-\infty,0)\},\\
\Phi^i_S(H)&=\Phi_S(T)\setminus(\Phi^+_S(H)\cup\Phi^-_S(H)),
\end{align*}
where $\sigma_{ess}(T)=\{\lambda\in\mathbb{R}:\, \lambda-T\notin\Phi_S(H)\}$ denotes the essential spectrum of $T\in\Phi_S(H)$. The components $\Phi^\pm_S(H)$ are contractible as topological spaces and on them the spectral flow of a path $L=\{L_\lambda\}_{\lambda\in[a,b]}$ is given by

\begin{align}\label{MorseRel}
\begin{split}
\sfl(L)&=\mu_{Morse}(L_a)-\mu_{Morse}(L_b)\,\,\,\text{if}\,\,\, L_\lambda\in\Phi^+_S(H),\\
\sfl(L)&=\mu_{Morse}(-L_b)-\mu_{Morse}(-L_a)\,\,\,\text{if}\,\,\, L_\lambda\in\Phi^-_S(H).
\end{split} 
\end{align}
A proof of these formulas can be found, e.g., in \cite[Prop. 4.3.1]{book} or \cite{CompSfl}. By the properties (S2) and (S3), the spectral flow induces a homomorphism between $\pi_1(\Phi^i_S(H))$ and the integers, which is actually an isomorphism by \cite{AtiyahPatodi}:   

\begin{itemize}
 \item[(S5)] The spectral flow induces an isomorphism
 \[\sfl:\pi_1(\Phi^i_S(H))\rightarrow\mathbb{Z}.\]
\end{itemize}
Pejsachowicz, Fitzpatrick and Recht showed in \cite{Specflow} that the spectral flow can be used to find bifurcation of critical points from a trivial branch for one-parameter families of functionals. Their main theorem is as follows.

\begin{theorem}\label{thm-bifJac}
Let $f:[a,b]\times H\rightarrow\mathbb{R}$ be a family of functionals such that every $f_\lambda:H\rightarrow\mathbb{R}$ is $C^2$ and its first and second derivatives depend continuously on $(\lambda,u)\in [a,b]\times H$. Assume that $(\nabla f_\lambda)(0)=0$, $\lambda\in [a,b]$ and that the Hessians $L_\lambda:=D^2_0f_\lambda$ at the critical point $0\in H$ are all Fredholm operators. If $L_a,L_b\in\GL(H)$ and

\[\sfl(L,[a,b])\neq 0,\]
then there is a bifurcation of critical points from the trivial branch $[a,b]\times\{0\}\subset [a,b]\times H$, i.e., there is some $\lambda^\ast\in(a,b)$ such that in every neighbourhood $U$ of $(\lambda^\ast,0)\in [a,b]\times H$, there is some $(\lambda,u)$ such that $u\neq 0$ is a critical point of $f_\lambda$.     
\end{theorem}
\noindent
Note that by \eqref{MorseRel} the classical theorem which shows that a change in the Morse index yields bifurcation if the latter are defined, is an immediate corollary of Theorem \ref{thm-bifJac}. Also other well known theorems in variational bifurcation theory can be obtained from Theorem \ref{thm-bifJac} as explained in the introduction of \cite{Specflow}. However, global bifurcation has never been studied by the spectral flow and it follows from a well-known example of B\"ohme \cite{Boehme} that there is no hope to get results in an abstract generality as in Theorem \ref{thm-bifJac}, i.e., there are examples of equations where a change of the Morse index only yield local but no global bifurcation.


\subsection{Spectral Flow, Parity and Global Bifurcation}
Note that if $H$ is a Hilbert space, then $\Phi_S(H)\subset\Phi_0(H)$ by the well-known fact that kernels of selfadjoint operators are perpendicular to their range. Consequently the parity and the spectral flow are both defined for paths in $\Phi_S(H)$. The following proposition explains the relation between these two invariants and a proof of it was outlined by Fitzpatrick, Pejsachowicz and Recht in \cite[Prop. 3.10]{Specflow}. Their argument is pretty much based on their particular construction of the spectral flow and uses Galerkin approximations as well as the so called cogredient parametrix that they introduced to transform paths in $\Phi_S(H)$ into some normal form. Here we propose a different proof in the spirit of \cite{Lesch} and \cite{CompSfl} that uses $(P5)$ and $(S5)$ and which is independent of the particular construction of the spectral flow.  

\begin{prop}\label{prop-parity}
Let $L=\{L_\lambda\}_{\lambda\in [a,b]}$ be a path in $\Phi_S(H)$ such that $L_a,L_b\in\GL(H)$. Then

\begin{align}\label{equ-sflpar}
\sfl(L,[a,b])\mod 2=\sigma(L,[a,b])\in\mathbb{Z}_2.
\end{align}
\end{prop}

\begin{proof}
Let us first note that we only need to show \eqref{equ-sflpar} for paths in $\Phi^i_S(H)$. Indeed, if $L=\{L_\lambda\}_{\lambda\in[a,b]}$ is a path in $\Phi^\pm_S(H)$, we consider $L^0=\{L^0_\lambda\}_{\lambda\in[a,b]}$, where $L^0_\lambda:H\oplus H\oplus H\rightarrow H\oplus H\oplus H$ is defined by $L^0_\lambda(v,u,w)=v+L_\lambda u-w$. As $H$ is of infinite dimension, it follows that $L^0_\lambda\in\Phi^i_S(H)$, and moreover we obtain from (P4) and (S4) that $\sigma(L^0,[a,b])=\sigma(L,[a,b])$ as well as $\sfl(L^0,[a,b])=\sfl(L,[a,b])$.\\
As next step, we recall the well-known fact that $\GL(H)\cap\Phi^i_S(H)$ is path connected, which can be shown by functional calculus as follows. Let $P_0$ be some proper orthogonal projection, i.e., the kernel and range of $P_0$ are both of infinite dimension. For $T\in\GL(H)\cap\Phi^i_S(H)$, we set $T_s=T|T|^{-1}((1-s)|T|+sI_H)$, where $|T|=\sqrt{T^2}$ and $s\in[0,1]$. As $|T|$ is positive, invertible and commutes with $T$, $T_s$ is in $\GL(H)\cap\Phi^i_S(H)$ for all $s\in[0,1]$ and connects $T$ to $T_1:=T|T|^{-1}$. The operator $T_1$ is orthogonal and selfadjoint and thus a symmetry, i.e., $T_1=T^\ast_1=T^{-1}_1$. Consequently,  $T_1=2P-I_H$ for an orthogonal projection which is proper as $T_1\in\Phi^i_S(H)$. Now as $P$ and $P_0$ are proper, there is an orthogonal operator $O\in\mathcal{L}(H)$ such that $P=O^\ast P_0O$ (cf. e.g. \cite[Prop. 5.1.7]{book}). This already shows that $\GL(H)\cap\Phi^i_S(H)$ is path connected since every orthogonal operator on an infinite dimensional Hilbert space can be connected by a path of orthogonal operators to the identity $I_H$, which was shown by Putnam and Wintner in \cite{Putnam} and also follows from Kuiper's Theorem.\\
Let now $\Omega_0(\Phi^i_S(H))$ be the set of all paths in $\Phi^i_S(H)$ having invertible endpoints. Then the parity and the $\mod 2$-reduction of the spectral flow, which we henceforth denote by $\sfl_2$, define maps

\begin{align*}
\sigma,\sfl_2 :\Omega_0(\Phi^i_S(H))\rightarrow\mathbb{Z}_2
\end{align*}
that induce homomorphisms between $\pi_1(\Phi^i_S(H))$ and $\mathbb{Z}_2$. As the fundamental group of $\Phi^i_S(H)$ is infinitely cyclic and $\Hom(\mathbb{Z},\mathbb{Z}_2)=\mathbb{Z}_2$, there only is one non-trivial homomorphism between these groups. Now by (S5), $\sfl_2$ is non-trivial on $\pi_1(\Phi^i_S(H))$. Moreover, a closed path in $\Phi^i_S(H)$ with non-trivial parity can be constructed as follows (cf. \cite[Theorem 3.1]{RobertIndBundle}). Let $\{e_k\}_{k\in\mathbb{Z}}$ be a complete orthonormal system of $H$. We denote by $P_0$ the orthogonal projection onto the span of $e_0$ and by $P_\pm$ the orthogonal projections onto the closures of the spans of $\{e_k\}_{\pm k\in\mathbb{N}}$, respectively. Consider the paths of operators
\[L_\lambda=P_+-P_-+\lambda\, P_0,\quad \lambda\in[-1,1],\]
and the constant path $M_\lambda=P_+-P_-+P_0\in GL(H)$. Then
\[M_\lambda L_\lambda=I_H-(1-\lambda)P_0,\quad\lambda\in[-1,1],\]
and thus $M$ is a parametrix for $L$. It follows from \eqref{LSdeg} that
\begin{align*}
\deg_{LS}(ML_1)=\deg_{LS}(I_H)=1 \text{ and } \deg_{LS}(M L_{-1})=\deg_{LS}(I_H-2P_0)=-1,
\end{align*}
which means that $\sigma(L,[-1,1])=-1$. As $\GL(H)\cap\Phi^i_S(H)$ is path connected, we can concatenate $L$ and a path in the latter set that connects the endpoint of $L$ to the initial point of $L$. By $(P1)$ and $(P2)$ this yields a closed path in $\Phi^i_S(H)$ having a non-trivial parity. Consequently, $\sigma$ and $\sfl_2$ coincide on $\pi_1(\Phi^i_S(H))$.\\
Now let $L$ be an element of $\Omega_0(\Phi^i_S(H))$ and let $L'$ be a path of operators in $\GL(H)\cap\Phi^i_S(H)$ that connects the endpoint of $L$ to its initial point. Then by (P1), (P2), (S1) and (S2)

\begin{align*}
\sigma(L)=\sigma(L')+\sigma(L)=\sigma(L'\ast L)=\sfl_2(L'\ast L)=\sfl_2(L')+\sfl_2(L)=\sfl_2(L),
\end{align*}
and thus $\sigma$ and $\sfl_2$ coincide on $\Omega_0(\Phi^i_S(H))$, which proves the proposition.
\end{proof}
\noindent
The following theorem is an immediate consequence of Theorem \ref{Rabier} and Proposition \ref{prop-parity}. We are not aware of any reference for the theorem, but guess that its statement is folklore. 

\begin{theorem}\label{thm-globalbif}
Let $f:\mathbb{R}\times H\rightarrow\mathbb{R}$ be $C^2$ and such that $\nabla f_\lambda=L_\lambda+C(\lambda,\cdot)$ for a path $L=\{L_\lambda\}_{\lambda\in \mathbb{R}}$ in $\Phi_S(H)$ and a compact map $C:\mathbb{R}\times H\rightarrow H$ such that $C(\lambda,0)=0$ as well as $D_0C(\lambda,\cdot)=0$ for all $\lambda\in\mathbb{R}$. Let $\mathcal{S}$ denote the closure of $$\{(\lambda,u)\in\mathbb{R}\times H:\, (\nabla f_\lambda)(u)=0\}\setminus\mathbb{R}\times\{0\}$$ in $\mathbb{R}\times H$, $\lambda_-<\lambda_+$ and let $\mathcal{T}$ be the connected component of $\mathcal{S}\cup[\lambda_-,\lambda_+]\times\{0\}$ containing $[\lambda_-,\lambda_+]\times\{0\}$. If $L_{\lambda_\pm}\in\GL(H)$ and $\sfl(L,[\lambda_-,\lambda_+])$ is odd, then there is a global bifurcation from the interval $[\lambda_-,\lambda_+]$, i.e.,  either $\mathcal{T}$ is unbounded or $\mathcal{T}$ contains a point $(\lambda^\ast,0)$ such that $\lambda^\ast\notin[\lambda_-,\lambda_+]$. 
\end{theorem}

\begin{proof}
By the remark below Theorem \ref{Rabier} we only need to show that $\nabla f=L+C:\mathbb{R}\times H\rightarrow H$ is proper on any bounded subset of $\mathbb{R}\times H$. As in the construction of the parity in Section \ref{sec-parity}, there is a parametrix for $L$, i.e., a continuous map $T:\mathbb{R}\rightarrow\GL(H)$ such that $T_\lambda L_\lambda=I_H+K_\lambda$ for some family of compact maps $K=\{K_\lambda\}_{\lambda\in\mathbb{R}}$. Now if $A\subset H$ is compact and $(\lambda_n,x_n)_{n\in\mathbb{N}}$ a sequence in $(\nabla f)^{-1}(A)\subset\mathbb{R}\times H$, then

\[L_{\lambda_n}x_n+C(\lambda_n,x_n)=:y_n\in A,\]
which yields

\begin{align*}
x_n=T_{\lambda_n}y_n-K_{\lambda_n}x_n-T_{\lambda_n}C(\lambda_n,x_n).
\end{align*}  
If $(\lambda_n,x_n)_{n\in\mathbb{N}}$ is bounded, then $(\lambda_n)_{n\in\mathbb{N}}$ has a convergent subsequence. The continuity of $T$ as well as the compactness of the set $A$ and the operators $K$ and $C$ now shows the existence of a convergent subsequence of $(\lambda_n,x_n)_{n\in\mathbb{N}}$. Thus the intersection of $(\nabla f)^{-1}(A)$ with any bounded subset is compact as a closed and relatively compact subset of $\mathbb{R}\times H$.
\end{proof}

\subsection{Proof of the Main Theorem}
Let us firstly note that by the compactness of the embedding $H^\frac{1}{2}(S^1,\mathbb{R}^{2n})\hookrightarrow L^2(S^1,\mathbb{R}^{2n})$, the selfadjoint operators $L_\lambda$ are Fredholm, i.e. $L_\lambda\in\Phi^i_S(H^\frac{1}{2}(S^1,\mathbb{R}^{2n}))$ (see \cite[Lemma 3.1]{CalcVar}). Moreover, $\nabla\psi_\lambda=L_\lambda+C(\lambda,\cdot)$, where $C:\mathbb{R}\times H^\frac{1}{2}(S^1,\mathbb{R}^{2n})\rightarrow H^\frac{1}{2}(S^1,\mathbb{R}^{2n})$ is compact according to \cite[Prop. B 37]{Rabinowitz}. Thus we are in the setting of Theorem \ref{thm-bifJac} and Theorem \ref{thm-globalbif}. Let us also note that it is sufficient to show the existence of local and global bifurcation in $H^\frac{1}{2}(S^1,\mathbb{R}^{2n})$ as every weak solution of \eqref{equation} is in fact classical (see, e.g., \cite[\S 12.3]{book}).\\
To show $(i)$, we firstly recall the comparison property of the spectral flow, which was shown by Pejsachowicz and the third author in \cite{BifJac}.

\begin{theorem}\label{thm-comparison}
Let $a,b\in\mathbb{R}$, $a<b$, and $\mathcal{L}=\{\mathcal{L}_\lambda\}_{\lambda\in [a,b]}$ as well as $\mathcal{M}=\{\mathcal{M}_\lambda\}_{\lambda\in [a,b]}$ paths in $\Phi_S(H)$ such that the difference $\mathcal{L}_\lambda-\mathcal{M}_\lambda$ is compact for all $\lambda\in [a,b]$. If
		
\begin{align}\label{order}
\mathcal{L}_a\geq \mathcal{M}_a\quad\text{and}\quad \mathcal{L}_b\leq \mathcal{M}_b,
\end{align}
then
		
\[\sfl(\mathcal{L},[a,b])\leq\sfl(\mathcal{M},[a,b]).\] 
	\end{theorem} 
\noindent
Here $\leq$ in \eqref{order} denotes the common partial order on the set of all selfadjoint operators given by
	
	\[T\leq S \Longleftrightarrow \langle(S-T)u,u\rangle\geq 0,\, u\in H.\] 
We now introduce in the setting of Theorem \ref{thm-main} two paths $M=\{M_\lambda\}_{\lambda\in[\lambda_-,\lambda_+]},N=\{N_\lambda\}_{\lambda\in[\lambda_-,\lambda_+]}$ of totally indefinite selfadjoint Fredholm operators in $H^\frac{1}{2}(S^1,\mathbb{R}^{2n})$ by

\begin{align}\label{MN}
\begin{split}
\langle M_\lambda u,v\rangle_{H^\frac{1}{2}}&=Q(u,v)+\int^{2\pi}_0{\langle B_\lambda u(t),v(t)\rangle\,dt}\\
\langle N_\lambda u,v\rangle_{H^\frac{1}{2}}&=Q(u,v)+\int^{2\pi}_0{\langle C_\lambda u(t),v(t)\rangle\,dt},
\end{split}
\end{align}
where $Q$ is the bilinear form introduced in \eqref{Gamma},

\begin{align}\label{BC}
\begin{split}
B_\lambda&=\beta_{-}I_{2n}+\frac{\lambda-\lambda_-}{\lambda_+-\lambda_-}(\alpha_{+}-\beta_{-})I_{2n},\\
C_\lambda&=\alpha_{-}I_{2n}+\frac{\lambda-\lambda_-}{\lambda_+-\lambda_-}(\beta_{+}-\alpha_{-})I_{2n}
\end{split}
\end{align}
are straight-line paths in $\Mat(2n,\mathbb{R})$ and $\alpha_\pm$ and $\beta_\pm$ are as defined in \eqref{alpha} and \eqref{beta}. Note that by the definition of $\alpha_\pm$ and $\beta_\pm$ we immediately see that

\[M_{\lambda_-}\geq L_{\lambda_-},\quad M_{\lambda_+}\leq L_{\lambda_+}\quad\text{as well as}\quad N_{\lambda_-}\leq L_{\lambda_-},\quad N_{\lambda_+}\geq L_{\lambda_+}. \]
As $M_\lambda-L_\lambda$ as well as $N_\lambda-L_\lambda$ are compact by the Rellich embedding theorem (see again \cite[Lemma 3.1]{CalcVar}), we obtain from Theorem \ref{thm-comparison}

\[\sfl(M,[\lambda_-,\lambda_+])\leq\sfl(L,[\lambda_-,\lambda_+])\leq \sfl(N,[\lambda_-,\lambda_+])\]  
and thus by Theorem \ref{thm-bifJac} we only need to show that $\sfl(M,[\lambda_-,\lambda_+])>0$ or $\sfl(N,[\lambda_-,\lambda_+])<0$ under the given assumptions.\\
Henceforth we make use of a method to compute the spectral flow that was invented by Robbin and Salamon in \cite{Robbin-Salamon}. Let $H$ be a separable Hilbert space and $\mathcal{L}=\{\mathcal{L}_\lambda\}_{\lambda\in [a,b]}$ a path in the normed space of bounded operators $\mathcal{L}(H)$, which we assume to be continuously differentiable with respect to the operator norm. We also assume that each $\mathcal{L}_\lambda$ is selfadjoint and Fredholm. A parameter value $\lambda^\ast$ is called a \textit{crossing} of $\mathcal{L}$ if $\ker(\mathcal{L}_{\lambda^\ast})\neq \{0\}$, and the \textit{crossing form} of a crossing is the quadratic form on the finite dimensional space $\ker(\mathcal{L}_{\lambda^\ast})$ defined by

\[\Gamma(\mathcal{L},\lambda^\ast)[u]=\langle \dot{\mathcal{L}}_{\lambda^\ast}u,u\rangle,\qquad u\in \ker(\mathcal{L}_{\lambda^\ast}),\]   
where $\dot{\mathcal{L}}_{\lambda^\ast}$ denotes the derivative of the path $\mathcal{L}$ at $\lambda=\lambda^\ast$. A crossing is called \textit{regular} if $\Gamma(\mathcal{L},\lambda^\ast)$ is non-degenerate. The following theorem was shown in \cite{Robbin-Salamon} (see also \cite{Homoclinics}).

\begin{theorem}\label{thm-sfl-crossings}
Let $\mathcal{L}=\{\mathcal{L}_\lambda\}_{\lambda\in[a,b]}$ be a path in $\Phi_S(H)$ as above.
\begin{enumerate}
\item[(a)] The crossings of $\mathcal{L}$ are isolated if they are regular, i.e., if $\lambda^\ast$ is regular, then $\mathcal{L}_\lambda$ is invertible for any $\lambda\neq\lambda^\ast$ that is sufficiently close to $\lambda$,
\item[(b)] If $\mathcal{L}$ has only regular crossings, then there are only finitely many crossings and the spectral flow of $\mathcal{L}$ can be computed by

\begin{align}\label{sfl-formula}
\sfl(\mathcal{L},[a,b])=-m^-(\Gamma(\mathcal{L},a))+\sum_{\lambda\in(a,b)}\sgn(\Gamma(\mathcal{L},\lambda))+m^-(-\Gamma(\mathcal{L},b)),
\end{align} 
where $\sgn$ denotes the signature and $m^-$ the Morse-index of a quadratic form.
\item[(c)] there is some $\varepsilon>0$ such that the perturbed path $\mathcal{L}^\delta=\{\mathcal{L}_\lambda+\delta I_H\}_{\lambda\in[a,b]}$ has only regular crossings for almost every $\delta\in(-\varepsilon,\varepsilon)$.
\end{enumerate}
\end{theorem}
\noindent
Now \eqref{sfl-formula} in Theorem \ref{thm-sfl-crossings} yields a rather straightforward proof of $(i)$ in Theorem \ref{thm-main}. We just need to note that 

\begin{align*}
\Gamma(M,\lambda)[u]&=\frac{1}{\lambda_+-\lambda_-}\int^{2\pi}_0{\langle (\alpha_+-\beta_-) u(t),u(t)\rangle\,dt}\\
\Gamma(N,\lambda)[u]&=\frac{1}{\lambda_+-\lambda_-}\int^{2\pi}_0{\langle (\beta_+-\alpha_-) u(t),u(t)\rangle\,dt}.
\end{align*}
As $\beta_-<\alpha_+$ and $\beta_+<\alpha_-$ by assumption, it follows that $\Gamma(M,\lambda)$ is positive definite on $\ker(M_\lambda)$, whereas $\Gamma(N,\lambda)$ is negative definite on $\ker(N_\lambda)$. Thus there is a bifurcation for \eqref{equation} if either $\ker(M_\lambda)\neq\{0\}$ or $\ker(N_\lambda)\neq 0$ for some $\lambda\in[\lambda_-,\lambda_+]$. Those kernels are given by the solutions of the equations

\[\left\{
\begin{aligned}
J u'(t)&+B_\lambda u(t)=0,\quad t\in [0,2\pi]\\
u(0)&=u(2\pi),
\end{aligned}
\right.\qquad \left\{
\begin{aligned}
J u'(t)&+C_\lambda u(t)=0,\quad t\in [0,2\pi]\\
u(0)&=u(2\pi),
\end{aligned}
\right.\]
and now it follows from \eqref{BC} by a simple computation that there is a non-trivial kernel if there is some $\lambda\in[\lambda_-,\lambda_+]$ such that either $B_\lambda$ or $C_\lambda$ is an integral multiple of the identity matrix $I_{2n}$. The latter happens for the path $B$ if and only if $\Delta(\beta_-,\alpha_+)> 0$ and for $C$ if and only if $\Delta(\alpha_-,\beta_+)< 0$. This finishes the proof of $(i)$. Let us note for later reference that we actually have shown by Theorem \ref{thm-sfl-crossings} that

\begin{align}\label{sfldelta}
\sfl(M,[\lambda_-,\lambda_+])=2n\Delta(\beta_-,\alpha_+),\quad\text{and}\quad \sfl(N,[\lambda_-,\lambda_+])=2n\Delta(\alpha_-,\beta_+).
\end{align}
Next we consider $(ii)$ in Theorem \ref{thm-main} and firstly define a homotopy $h:[0,1]\times[\lambda_-,\lambda_+]\rightarrow\Phi^i_S(H)$ by

\[\langle h(s,\lambda)u,v\rangle=Q(u,v)+\int^{2\pi}_0\langle ((1-s)A_\lambda(t)+sB_\lambda)u(t),v(t)\rangle dt.\]
Note that indeed $h(s,\lambda)\in\Phi^i_S(H)$, as $h(s,\lambda)-L_\lambda$ is compact for any $(s,\lambda)\in [0,1]\times[\lambda_-,\lambda_+]$. By the homotopy invariance of the spectral flow (S3), it follows that the spectral flow of the loop made by restricting $h$ to the boundary of the square $[0,1]\times[\lambda_-,\lambda_+]$ vanishes. Thus we obtain by the concatenation property (S2) that

\begin{align*}
\sfl(L,[\lambda_-,\lambda_+])&=\sfl(h(\cdot,\lambda_-),[0,1])+\sfl(h(1,\cdot),[\lambda_-,\lambda_+])-\sfl(h(\cdot,\lambda_+),[0,1])\\
&=\sfl(h(\cdot,\lambda_-),[0,1])+\sfl(M,[\lambda_-,\lambda_+])-\sfl(h(\cdot,\lambda_+),[0,1]),
\end{align*}   
where the negative sign in the final term comes from reversing the orientation of the path and is a direct consequence of the definition \eqref{def-sfl}. We now aim to show that the latter sum is strictly greater than $0$ under the given assumptions. The following auxiliary result was shown in \cite[\S 7]{BifJac}. Here we include a new and pretty straightforward proof for the convenience of the reader.

\begin{prop}\label{prop-pos}
Let $\mathcal{L}=\{\mathcal{L}_\lambda\}_{\lambda\in[a,b]}$ be a path in $\Phi_S(H)$ that is $C^1$ with respect to the norm of $\mathcal{L}(H)$. If $\mathcal{L}_\lambda\geq\mathcal{L}_\mu$ for all $\lambda,\mu\in[a,b]$, $\lambda\geq\mu$, then $\sfl(\mathcal{L},[a,b])\geq 0$.
\end{prop} 

\begin{proof}
As $\mathcal{L}_\lambda\geq\mathcal{L}_\mu$ for all $\lambda,\mu\in[a,b]$, $\lambda\geq \mu$, we first note that $\Gamma(\mathcal{L},\lambda)[u]=\langle\dot{\mathcal{L}}_\lambda u,u\rangle\geq 0$ for all $\lambda\in[a,b]$. It follows from the definition of the spectral flow in \eqref{def-sfl} that $\sfl(\mathcal{L},[a,b])=\sfl(\mathcal{L}^\delta,[a,b])$ for any sufficiently small $\delta>0$, where $\mathcal{L}^\delta=\{\mathcal{L}_\lambda+\delta I_H\}_{\lambda\in[a,b]}$ (see \cite[Lemma 2.1]{IJW}). By Theorem \ref{thm-sfl-crossings} (c) we now can choose such a $\delta>0$ in a way that moreover $\mathcal{L}^\delta$ only has regular crossings. Now $\Gamma(\mathcal{L}^\delta,\lambda)=\Gamma(\mathcal{L},\lambda)$ and as the latter is positive semi-definite and non-degenerate on every $\ker(\mathcal{L}_\lambda+\delta I_H)$, we finally see from \eqref{sfl-formula} that $\sfl(\mathcal{L},[a,b])=\sfl(\mathcal{L}^\delta,[a,b])\geq 0$. 
\end{proof}
\noindent
Now we firstly note that by the previous proposition

\[\sfl(M,[\lambda_-,\lambda_+])\geq 0\quad\text{and}\quad -\sfl(h(\cdot,\lambda_+),[0,1])\geq 0,\]
and thus by Theorem \ref{thm-bifJac} it is sufficient to show that $\sfl(h(\cdot,\lambda_-),[0,1])>0$. The path $h(\cdot,\lambda_-)$ is given by 

\[\langle h(s,\lambda_-)u,v\rangle=Q(u,v)+\int^{2\pi}_0\langle ((1-s)A_{\lambda_-}+sB_{\lambda_-})u(t),v(t)\rangle dt,\]
and thus 

\begin{align}\label{crossing1}
\Gamma(h(\cdot,\lambda_-),s)[u]=\int^{2\pi}_0{\langle (B_{\lambda_-}-A_{\lambda_-})u(t),u(t)\rangle dt}.
\end{align}
To proceed, we firstly need to make a brief digression and recall that crossing forms were originally introduced by Robbin and Salamon in \cite{Robbin-Salamon} for computing the spectral flow of paths of unbounded selfadjoint operators $\mathcal{A}=\{\mathcal{A}_\lambda\}_{\lambda\in[a,b]}$, where each $\mathcal{A}_\lambda$ is defined on a dense domain $D\subset H$ which is moreover supposed to be a Hilbert space in its own right that is compactly embedded into $H$. Theorem \ref{thm-sfl-crossings} verbatim holds in this setting (see \cite{Homoclinics}) and we now aim to apply it to the path of differential operators

\[\mathcal{A}_s:H^1(S^1,\mathbb{R}^{2n})\subset L^2(S^1,\mathbb{R}^{2n})\rightarrow L^2(S^1,\mathbb{R}^{2n}),\quad (\mathcal{A}_s u)(t)=J u'(t)+((1-s)A_{\lambda_-}+sB_{\lambda_-}) u(t). \]
As already used in the proof of Proposition \ref{prop-pos}, the spectral flow is invariant under perturbations by a sufficiently small positive multiple of the identity (see \cite[Lemma 2.1]{IJW}). Consequently, for almost every sufficiently small $\delta>0$, we have that $\sfl(\mathcal{A},[0,1])=\sfl(\mathcal{A}^\delta,[0,1])$, where the path $\mathcal{A}^\delta=\{\mathcal{A}_s+\delta I_H\}_{s\in[0,1]}$ has only regular crossings and its spectral flow is given by \eqref{sfl-formula}. As $B_{\lambda_-}=\beta_-I_{2n}$, we see that 

\begin{align}\label{crossing2}
\Gamma(\mathcal{A},s)=\Gamma(\mathcal{A}^\delta,s)=\int^{2\pi}_0{\langle (B_{\lambda_-}-A_{\lambda_-})u(t),u(t)\rangle dt}
\end{align}
is positive semi-definite. Now the kernel of $\mathcal{A}_s$ is given by the solutions of the boundary value problem

\[\left\{
\begin{aligned}
J u'(t)&+((1-s)A_{\lambda_-}+sB_{\lambda_-}) u(t)=0,\quad t\in [0,2\pi]\\
u(0)&=u(2\pi).
\end{aligned}
\right.\]
As $\alpha_-<0<\beta_-$ by assumption, we see that there has to be some $s\in(0,1)$ such that the matrix $(1-s)A_{\lambda_-}+sB_{\lambda_-}$ has a non-trivial kernel, which shows that the differential operator $\mathcal{A}_s$ has a non-trivial kernel by a constant solution. As $B_{\lambda_-}$ is a scalar multiple of the identity matrix $I_{2n}$ it follows that also $\mathcal{A}^\delta_s$ has a non-trivial kernel for any sufficiently small $\delta>0$ for some $s\in(0,1)$. This ultimately shows that 

\begin{align}\label{sflg0}
\sfl(\mathcal{A},[0,1])=\sfl(\mathcal{A}^\delta,[0,1])>0,
\end{align}
where $\delta>0$ is moreover chosen in a way such that $\mathcal{A}^\delta$ has only regular crossings.\\
Now we come back to the path $h(\cdot,\lambda_-)$ in our homotopy and firstly note that $h(0,\lambda_-)$ is invertible by the assumptions of Theorem \ref{thm-main}. The operator $h(1,\lambda_-)$ is invertible if and only if $\beta_-$ is not an integer. If $\beta_-$ is an integer, we consider instead $\beta_-+\varepsilon$ for an arbitrary $0<\varepsilon<1$ such that $\beta_-+\varepsilon<\alpha_+$. This does not affect the previous arguments and thus we can assume without loss of generality that the path $h(\cdot,\lambda_-)$ has invertible endpoints. By the homotopy invariance of the spectral flow $(S3)$, we have that $\sfl(h(\cdot,\lambda_-),[0,1])=\sfl(\tilde{h},[0,1])$, where

\[\langle \tilde{h}_su,v\rangle=Q(u,v)+\int^{2\pi}_0\langle ((1-s)A_{\lambda_-}+sB_{\lambda_-}+\delta I_{2n})u(t),v(t)\rangle dt\]
and $\delta>0$ is sufficiently small. Now $\ker(\tilde{h}_s)=\ker(\mathcal{A}^\delta_s)$ and $\Gamma(\tilde{h},s)=\Gamma(\mathcal{A}^\delta,s)$ for any $s\in[0,1]$. Thus if $\delta>0$ is suitably chosen such that $\mathcal{A}^\delta$ only has regular crossings, then this also is the case for $\tilde{h}$ and we finally obtain from \eqref{sfl-formula} that

\[\sfl(h(\cdot,\lambda_-),[0,1])=\sfl(\tilde{h},[0,1])=\sfl(\mathcal{A}^\delta,[0,1]),\]
which proves the first case in $(ii)$ by \eqref{sflg0}.\\
Now we assume that $\beta_+<\alpha_-$ and consider the path $N=\{N_\lambda\}_{\lambda\in[\lambda_-,\lambda_+]}$. As above, the homotopy  $h:[0,1]\times[\lambda_-,\lambda_+]\rightarrow\Phi^i_S(H)$ given by

\[\langle h(s,\lambda)u,v\rangle=Q(u,v)+\int^{2\pi}_0\langle ((1-s)A_\lambda(t)+sC_\lambda)u(t),v(t)\rangle dt\]
shows that

\begin{align}\label{compN}
\begin{split}
\sfl(L,[\lambda_-,\lambda_+])&=\sfl(h(\cdot,\lambda_-),[0,1])+\sfl(h(1,\cdot),[\lambda_-,\lambda_+])-\sfl(h(\cdot,\lambda_+),[0,1])\\
&=\sfl(h(\cdot,0),[0,\lambda_-])+\sfl(N,[\lambda_-,\lambda_+])-\sfl(h(\cdot,1),[0,\lambda_+]).
\end{split}
\end{align}
Now the rest of the argument is very similar to the case of $M$. By Proposition \ref{prop-pos} it follows that 

\[\sfl(N,[\lambda_-,\lambda_+])\leq 0\quad\text{and}\quad -\sfl(h(\cdot,\lambda_+),[0,1])\leq 0.\]
Moreover, using crossing forms, we verbatim see that $\sfl(h(\cdot,\lambda_-),[0,1])<0$ if $\alpha_-<0<\beta_-$. Consequently, $\sfl(L,[\lambda_-,\lambda_+])<\sfl(N,[\lambda_-,\lambda_+])\leq 0$, which again shows the existence of a bifurcation point by Theorem \ref{thm-bifJac}.\\ 
It remains to show $(iii)$ on global bifurcation. Our aim is to firstly obtain a general estimate for the spectral flow of $L$ in any dimension $2n$ of \eqref{equation}, which then yields the announced global bifurcation for planar systems. We consider once again the paths $M$ and $N$ defined in \eqref{MN}, but we adjust in two consecutive steps the previously assumed relations between the numbers $\alpha_\pm$ and $\beta_\pm$.  From the proof of (ii), we already know that $0\leq\sfl(M,[\lambda_-,\lambda_+])<\sfl(L,[\lambda_-,\lambda_+])$ if $\beta_-<\alpha_+$. The latter and \eqref{Delta} imply that now $\beta_+>\alpha_-$. If we modify $N$ by assuming that $\beta_+>\alpha_-$ and repeat the argument in $(ii)$, then the only change in \eqref{compN} is that now $\sfl(N,[\lambda_-,\lambda_+])\geq 0$. Thus we obtain from \eqref{sfldelta} that

\[0\leq\sfl(M,[\lambda_-,\lambda_+])=2n\Delta(\beta_-,\alpha_+)<\sfl(L,[\lambda_-,\lambda_+])<2n\Delta(\alpha_-,\beta_+)=\sfl(N,[\lambda_-,\lambda_+]).\]  
If $n=1$, i.e. the system \eqref{equation} is planar, and $\Delta(\alpha_-,\beta_+)-\Delta(\beta_-,\alpha_+,)=1$ as required in \eqref{Delta}, then $\sfl(L,[\lambda_-,\lambda_+])$ is odd, which implies the existence of a global bifurcation from the interval $[\lambda_-,\lambda_+]$ by Theorem \ref{thm-globalbif}.\\
In a second step, we keep the original assumption $\beta_+<\alpha_-$ on $N$, but now modify $M$ by assuming that $\alpha_+> \beta_-$. If we repeat the above argument, this yields

\[\sfl(M,[\lambda_-,\lambda_+])=2n\Delta(\alpha_+,\beta_-)<\sfl(L,[\lambda_-,\lambda_+])<2n\Delta(\beta_+,\alpha_-)=\sfl(N,[\lambda_-,\lambda_+])\leq 0.\]      
Thus again we obtain in the planar case a global bifurcation if $\Delta(\beta_+,\alpha_-)-\Delta(\alpha_+,\beta_-)=1$. This finishes the proof of (iii).



\vspace*{1.3cm}

\begin{minipage}{1.2\textwidth}
\begin{minipage}{0.4\textwidth}

Joanna Janczewska\\
Institute of Applied Mathematics\\
Faculty of Applied Physics and Mathematics\\
Gda\'{n}sk University of Technology\\
Narutowicza 11/12, 80-233 Gda\'{n}sk, Poland\\
joanna.janczewska@pg.edu.pl\\\\\\
Maciej Starostka\\
Institute of Applied Mathematics\\
Faculty of Applied Physics and Mathematics\\
Gda\'{n}sk University of Technology\\
Narutowicza 11/12, 80-233 Gda\'{n}sk, Poland\\
maciej.starostka@pg.edu.pl\\\\

\end{minipage}
\hfill
\begin{minipage}{0.6\textwidth}

Nils Waterstraat\\
Martin-Luther-Universit\"at Halle-Wittenberg\\
Naturwissenschaftliche Fakult\"at II\\
Institut f\"ur Mathematik\\
06099 Halle (Saale)\\
Germany\\
nils.waterstraat@mathematik.uni-halle.de
\end{minipage}
\end{minipage}

\end{document}